\documentclass[11pt,reqno]{amsart}

\usepackage{a4wide}
\usepackage{dsfont} 
\usepackage{amssymb,amsmath}
\usepackage{mathrsfs}
\usepackage{comment}
\usepackage[active]{srcltx} 
\usepackage[all]{xy}

\newtheorem{proposition}{Proposition}[section]
  \newtheorem{theorem}[proposition]{Theorem}
  
  \newtheorem{corollary}[proposition]{Corollary}
\theoremstyle{remark}
  \newtheorem{definition}[proposition]{Definition}
  \newtheorem{remark}[proposition]{Remark}
   \newtheorem{question}[proposition]{Question}

\newcommand{\cst}{\ifmmode\mathrm{C}^*\else{$\mathrm{C}^*$}\fi}
\newcommand{\st}{\;\vline\;}
\newcommand{\CC}{\mathbb{C}}
\newcommand{\NN}{\mathbb{N}}
\newcommand{\tens}{\otimes}
\newcommand{\atens}{\otimes_{\text{\tiny{alg}}}} 
\newcommand{\id}{\mathrm{id}}
\newcommand{\comp}{\!\circ\!}
\newcommand{\I}{\mathds{1}}
\newcommand{\vt}{\!\vartriangle\!}
\newcommand{\GG}{\mathbb{G}}
\newcommand{\HH}{\mathbb{H}}
\newcommand{\II}{\mathbb{I}}
\newcommand{\SSS}{\mathbb{S}}
\newcommand{\KK}{\mathbb{K}}
\newcommand{\sA}{\mathsf{A}}

\newcommand{\sB}{\mathsf{B}}
\newcommand{\sM}{\mathsf{M}}
\newcommand{\sC}{\mathsf{C}}
\newcommand{\sJ}{\mathsf{J}}
\newcommand{\sS}{\mathsf{S}}
\newcommand{\cT}{\mathcal{T}}
\newcommand{\is}[2]{\left(#1\,\vline\,#2\right)}
\newcommand{\bb}{\boldsymbol{b}}
\newcommand{\cDelta}{\overset{\circ}{\Delta}}
\newcommand{\sD}{\mathsf{D}}
\newcommand{\sK}{\mathsf{K}}
\renewcommand{\Bar}[1]{\overline{#1}}

\DeclareMathOperator{\C}{C}
\DeclareMathOperator{\M}{M}
\DeclareMathOperator{\Mor}{Mor}
\DeclareMathOperator{\Aut}{Aut}
\DeclareMathOperator{\Pol}{Pol}

\newenvironment{rlist}
{

\begin{enumerate}}
{\end{enumerate}}

\numberwithin{equation}{section}

\title{Quantum families of invertible maps and related problems}

\author{Adam Skalski}
\address{Institute of Mathematics of the Polish Academy of Sciences,
ul.~\'Sniadeckich 8, 00--956 Warszawa, Poland
\newline \indent Faculty of Mathematics, Informatics and Mechanics, University of Warsaw, ul.~Banacha 2,
02-097 Warsaw, Poland}
\thanks{AS  was partially supported by the NCN (National Centre of Science) grant
2014/14/E/ST1/00525}
\email{a.skalski@impan.pl}

\author{Piotr M.~So{\l}tan} \address{Department of Mathematical Methods in Physics, Faculty of Physics, University of Warsaw, Poland}
\email{piotr.soltan@fuw.edu.pl}

\keywords{Quantum families of invertible maps, Hopf image, universal quantum group}
\subjclass[2010]{ Primary 46L89, Secondary 46L65}

\begin{document}

\begin{abstract}
The notion of families of quantum invertible maps (\cst-algebra homomorphisms satisfying Podle\'s condition) is employed to strengthen and reinterpret several results concerning universal quantum groups acting on finite quantum spaces. In particular Wang's quantum automorphism groups are shown to be universal with respect to quantum families of invertible maps. Further the construction of the Hopf image of Banica and Bichon is phrased in the purely analytic language and employed to define the quantum subgroup generated by a family of quantum subgroups or more generally a family of quantum invertible maps.
\end{abstract}

\maketitle

\section{Introduction}

A theorem of Gelfand and Naimark, identifying the category of compact spaces (with continuous maps) and unital commutative \cst-algebras (with unital $*$-homomorphisms) has become an entry point for a very rich host of noncommutative generalizations and interpretations. The idea that general \cst-algebras should be viewed as `quantum' topological spaces was also behind the theory of compact, and later locally compact, quantum groups. As in the classical world, quantum groups are best seen `in action', i.e.\ as families of symmetries of a given space. The concept of `quantum families of maps' appeared already in late 1970s (\cite{pseu}): compact quantum groups were born in the second half of the 1980s, and the study of their actions on \cst-algebras  started soon after that (\cite{podPhd}). Finally in 1998 Shuzhou Wang proved in \cite{Wang} the existence of the universal compact quantum group acting on a finite-dimensional \cst-algebra $\sA$ with a fixed faithful state $\omega$. The objects introduced by Wang have proved to be fascinating and useful for several reasons: in the purely algebraic sense, as a source of new examples of Hopf $*$-algebras, analytically, yielding  \cst-algebras of highly non-trivial structure (\cite{ChristianKtheory}, \cite{MikeHAP}) and as an arena for classical and noncommutative geometry and probability, with several, often unexpected, connections to other areas of mathematics (see for example \cite{Hadamard}).

We begin this article by revisiting the original idea of Wang, connecting it to the notion of quantum families of \emph{invertible maps}. The latter are defined via the non-degeneracy type condition introduced by Podle\'s in \cite{podPhd} (see also [So$_{1-2}$]). We give an alternative proof of the main result of \cite{Wang}, showing at the same time  that the algebras constructed there are in fact universal with respect to all quantum families of invertible maps acting on a given $(\sA, \omega)$ (and not only for compact quantum group actions). Then we pass to the study of \emph{Hopf images}. The notion of a Hopf image $A_{\Lambda}$ for an algebra homomorphism $\Lambda$ from a Hopf algebra $A$ into another algebra, understood as the largest Hopf algebra through which $\Lambda$ factorises as an algebra homomorphism, was introduced by Banica and Bichon in \cite{BanicaBichon} and later studied from the compact quantum group point of view for example in \cite{HopfIdempotents}. Here (in Theorem 4.1) we provide a purely analytic approach to this concept in the compact quantum group setting. This allows us to further formalise the concept of the quantum group generated by a family of quantum subgroups (recently introduced also in \cite{BCV}), and by a given quantum family of invertible maps. This brings us back to the first part of the paper, at the same time offering an alternative point of view on \emph{inner linearity} of a compact quantum group $\GG$: this important notion originally introduced in \cite{BanicaBichon}, dual to the \emph{linearity} for discrete groups, is equivalent to the fact that $\GG$ is generated by a finite quantum family of invertible maps.

The plan of the article is as follows: in Section 2 we introduce the terminology of quantum spaces and maps and recall some known results. Section 3 presents quantum families of invertible maps and revisits the construction of the universal compact quantum group acting on a finite-dimensional \cst-algebra $\sA$ with a fixed state, due to Wang. In particular we show there that Wang's algebra is universal with respect to quantum families of invertible maps on $\sA$ fixing the relevant state. Section 4 passes to the study of Hopf images, phrasing the concept in the purely analytic language. Section 5 explains how to use the notion of the Hopf image to define the quantum group generated by a family of quantum subgroups, and Section 6 contains the construction of the quantum group generated by a quantum family of invertible maps, leading to a new interpretation of inner linearity. Finally in Section 7 we discuss briefly a generation question concerning the quantum increasing sequences of Curran, and in the Appendix present an alternative, simpler proof of Theorem 2.3, originally showed in \cite{qs} by other methods.

\textbf{Acknowledgment:} We thank the anonymous referee for a very careful reading of our manuscript.

\section{Quantum spaces and quantum families of maps}

\newcommand{\XX}{\mathbb{X}}
\newcommand{\YY}{\mathbb{Y}}
\newcommand{\EE}{\mathbb{E}}
\newcommand{\UU}{\mathbb{U}}
\newcommand{\FF}{\mathbb{F}}
\newcommand{\DD}{\mathbb{D}}

Consider the  \emph{category of \cst-algebras}, i.e.~the category whose objects are all \cst-algebras and for any two \cst-algebras  $\sA$ and $\sB$ the set of $\Mor(\sA,\sB)$ of morphisms from $\sA$ to $\sB$ is defined to be the set of all $*$-homomorphisms $\Phi$ from $\sA$ to the multiplier algebra $\M(\sB)$ such that $\Phi(\sA)\sB$ is dense in $\sB$ (\cite{pseu,gen}). The particular choice of morphisms ensures that the full subcategory formed by commutative \cst-algebras is dual to the category of locally compact (Hausdorff) spaces with continuous maps as morphisms. The duality is realized by the functor mapping a locally compact space $X$ to the \cst-algebra $\C_0(X)$ of all continuous functions on $X$ vanishing at infinity.

\begin{definition}
A \emph{quantum space} is an object of the category dual to the category of \cst-algebras.
\end{definition}

Strictly speaking any theorem about quantum spaces is nothing else but a theorem about \cst-algebras. Nevertheless the study of \cst-algebras ``as if they were algebras of functions'' has lead to many exciting developments. In particular the theory of compact and locally compact quantum groups originated through this approach to the theory of \cst-algebras.

\sloppy
The standard notation used in studying quantum spaces is the following: for a quantum space $\XX$ we write $\C_0(\XX)$ for the corresponding \cst-algebra. A \emph{map of quantum spaces} from a quantum space $\XX$ to another quantum space $\YY$ is then an element of the set $\Mor\bigl(\C_0(\YY),\C_0(\XX)\bigr)$. In this paper almost all quantum spaces will be \emph{compact}, i.e.~their corresponding \cst-algebras will be unital. It is only natural to write then $\C(\XX)$, $\C(\YY)$ etc.~to denote the corresponding \cst-algebras. Also morphisms form $\XX$ to $\YY$ are then simply unital $*$-homomorphism form $\C(\YY)$ to $\C(\XX)$. The symbol $\tens$ will denote the minimal/spatial tensor product of \cst-algebras, and the algebraic tensor product will be denoted by $\tens_{\textup{alg}}$.

Motivated by the classical characterization of continuous families of maps between topological spaces the following definition was introduced in \cite{pseu} (cf.~also \cite[Section 3]{qs}).

\begin{definition}\label{DefQFM}
Let $\XX$, $\YY$ and $\UU$ be quantum spaces. A \emph{quantum family of maps} from $\XX$ to $\YY$ indexed by $\UU$ is an element of $\Mor\bigl(\C_0(\YY),\C_0(\XX)\tens\C_0(\UU)\bigr)$.
\end{definition}

Definition \ref{DefQFM} is very broad and hence not much can be said about all quantum families of maps. However with appropriate notion of quantum families of \emph{invertible} maps (Definition \ref{DefInv}) we will show in Section \ref{Sgen} that they generate (in an appropriate sense) actions of quantum groups.

Once the definition of a quantum family of maps has been formulated, the question of existence of quantum families of maps with specific properties of universal nature becomes particularly important. An important theorem announced in \cite{pseu} states that certain universal quantum families always exist.

\begin{theorem}\label{existenceQFM}
Let $\XX$ and $\YY$ be quantum spaces such that $\C(\XX)$ is finite dimensional and $\C(\YY)$ is finitely generated and unital. Then there exists a compact quantum space $\EE$ and a quantum family of maps $\XX\to\YY$ indexed by $\EE$
\[
\Phi\colon\C(\YY)\to\C(\XX)\tens\C(\EE)
\]
such that for any quantum space $\UU$ and any quantum family of maps $\Psi\in\Mor\bigl(\C(\YY),\C(\XX)\tens\C_0(\UU)\bigr)$ there exists a unique $\Lambda\in\Mor\bigl(\C(\EE),\C_0(\UU)\bigr)$ such that $\Psi=(\id\tens\Lambda)\comp\Phi$ or, in  other words, the diagram
\[
\xymatrix
{
\C(\YY)\ar@{=}[d]\ar[rr]^{\Phi}&&\C(\YY)\tens\C(\EE)\ar[d]^{\id\tens\Lambda}\\
\C(\YY)\ar[rr]^\Psi&&\C(\YY)\tens\C(\UU)
}
\]
is commutative.

Moreover the unital \cst-algebra $\C(\EE)$ is  finitely generated.
\end{theorem}

The universal property of the quantum family of maps introduced in Theorem \ref{existenceQFM} justifies calling $\Phi$ the \emph{quantum family of all maps $\XX\to\YY$} and the quantum space $\EE$ the \emph{quantum space of all maps $\XX\to\YY$}. In \cite{qs} this existence result was proved and used to construct many other universal quantum families of maps which all carried additional algebraic structure.

\begin{remark}
\noindent\begin{enumerate}
\item The universal property defines the pair $\bigl(\C(\EE),\Phi\bigr)$ uniquely (up to isomorphism of such pairs).
\item The \cst-algebra $\C(\EE)$ is generated by the set
\[
\bigl\{(\omega\tens\id)\Phi(c)\st{c}\in\C(\YY),\;\omega\in\C(\XX)^*\bigr\}
\]
(cf.~\cite{qs} or the proof of Theorem \ref{existenceQFM} in the Appendix).
\end{enumerate}
\end{remark}

In the proof of Theorem \ref{existenceQFM} given in \cite{qs} the \cst-algebra $\C(\EE)$ was defined in terms of generators and relations. However arbitrary families of relations needed to be allowed which made the proof rather complicated. In the Appendix we will give a different proof of Theorem \ref{existenceQFM} which avoids these complications. For clarity let us stress that by a finitely generated unital \cst-algebra we mean a \cst-algebra isomorphic to a quotient of the full group \cst-algebra $\cst(\FF_n)$ of the free group on $n$ generators (for some $n\in \NN$).

In the paper \cite{qs} an operation on quantum families of maps was introduced which corresponds (in the classical --- commutative --- case) to the operation of constructing the family of all possible compositions of maps from two families.

\begin{definition}
Let $\XX_1,\XX_2,\XX_3,\DD_1$ and $\DD_2$ be quantum spaces. Consider quantum families of maps
\[
\Psi_1\in\Mor\bigl(\C_0(\XX_2),\C_0(\XX_1)\tens\C_0(\DD_1)\bigr)\quad\text{and}\quad\Psi_2\in\Mor\bigl(\C_0(\XX_3),\C_0(\XX_2)\tens\C_0(\DD_2)\bigr).
\]
The \emph{composition} of $\Psi_1$ and $\Psi_2$ is the quantum family of maps
\[
\Psi_1\vt\Psi_2\in\Mor\bigl(\C_0(\XX_3),\C_0(\XX_1)\tens\C_0(\DD_1)\tens\C_0(\DD_2)\bigr)
\]
defined by
\[
\Psi_1\vt\Psi_2=(\Psi_1\tens\id)\comp\Psi_2.
\]
\end{definition}

Thus the composition of $\Psi_1$ and $\Psi_2$ is a quantum family of maps $\XX_1\to\XX_3$ indexed by $\DD_1\times\DD_2$. Let us remark that composition of quantum families of maps is associative.

Once quantum families of maps are defined one is then lead to considering families with specified properties. One such property is preservation of a state defined as follows.

\begin{definition}
Let $\XX$ and $\YY$ be quantum spaces and let $\phi$ be a state on $\C(\XX)$. Let $\Phi\in\Mor\bigl(\C_0(\XX),\C_0(\XX)\tens\C_0(\YY)\bigr)$ be a quantum family of maps $\XX\to\XX$ indexed by $\YY$. We say that $\Phi$ \emph{preserves $\phi$} or that $\phi$ is \emph{invariant} for $\Phi$ if for any $x\in\C_0(\XX)$ we have
\[
(\phi\tens\id)\Phi(x)=\phi(x)\I_{\C(\YY)}.
\]
\end{definition}

Introduction of the notion of a quantum family of maps preserving a state lead to major developments. In particular invariance of states for actions of compact quantum groups (which are quantum families of maps, see Section \ref{InvQF}) turned out to be very important in e.g.~\cite{Wang} (cf.~also \cite[Section 5]{qs}). A crucial (although also very simple) fact is that if two quantum families of maps preserve a given state then so does their composition (\cite[Proposition 14]{qs}).

In the paper \cite{Wang} it was shown that for a finite quantum space $\XX$ and a state $\phi$ on $\C(\XX)$ there exists a universal compact quantum group acting on $\XX$ preserving the state $\phi$. Additionally it was shown that in general (more precisely for a noncommutative $\C(\XX)$) there is no universal quantum group for all actions of compact quantum groups. This fact often forces us to consider finite quantum spaces (finite dimensional \cst-algebras) together with a distinguished state. This will be also clearly visible throughout this paper.

\section{Quantum families of invertible maps}\label{InvQF}

In this section we will give an alternative description of Wang's quantum automorphism groups of finite dimensional \cst-algebras with a distinguished state. We will only consider the case of a faithful state which is justified by \cite[Proposition 2.3]{apqs}. In the original description of these quantum automorphism groups in \cite{Wang}  little explanation is given as to the origin of the proposed relations imposed on generators of corresponding \cst-algebras. In particular the case of non-tracial states is not treated (this case is treated thoroughly in \cite{WVD}, but with no reference to  actions of resulting quantum groups on finite dimensional \cst-algebras). Our description is similar to the one proposed in \cite[Lemma 1.2 \& Theorem 1.1]{symcoact}, but we choose to focus on the stronger universality property (cf.~Remark \ref{RemInv}\eqref{RemInv2}).

We begin by introducing the notion of a quantum family of invertible maps on a finite quantum space (i.e.~one whose corresponding \cst-algebra is finite dimensional). For the sake of keeping our notation lighter we will drop the custom to refer to any \cst-algebra as $\C(\XX)$ for some quantum space $\XX$ and use more usual symbols such as $\sA$, $\sB$ etc.

\begin{definition}\label{DefInv}
Let $\sM$ and $\sB$ be \cst-algebras with $\sB$ unital and $\sM$ finite dimensional. We say that
\[
\beta\colon\sM\longrightarrow\sM\tens\sB
\]
is a \emph{quantum family of invertible maps} if $\beta(\sM)(\I\tens\sB)$ is dense in $\sM\tens\sB$. We will also refer to this by saying that $\beta$ satisfies the \emph{Podle\'s condition}.
\end{definition}

The Podle\'s condition appeared already in the thesis of Podle\'s (\cite{podPhd}) and was lated put to use in many publications (cf.~\cite{podles,boca}). It is formally very close to the density conditions (\emph{cancellation laws}) central to the definition of a compact quantum group (\cite{cqg}).

Its original meaning is closely related to the notion of an action of a quantum group on a quantum space. For simplicity let us restrict attention to compact quantum groups and compact quantum spaces. If $\GG$ is a compact quantum group and $\XX$ a compact quantum space then an action of $\GG$ on $\XX$ is a unital $*$-homomorphism $\alpha\colon\C(\XX)\to\C(\XX)\tens\C(\GG)$ such that
\[
(\alpha\tens\id)\comp\alpha=(\id\tens\Delta_\GG)\comp\alpha
\]
and the Podle\'s condition is satisfied, i.e. $\alpha\bigl(\C(\XX)\bigr)\bigl(\I\tens\C(\GG)\bigr)$ is dense in $\C(\XX)\tens\C(\GG)$. In particular an action of $\GG$ on $\XX$ is an example of a quantum family of invertible maps. The Podle\'s condition serves as a substitute of the requirement imposed on actions of classical groups on sets, namely that the unit element act as the identity map or, equivalently, that the action be by invertible transformations (\cite[Proposition 2.3]{exa}).

\begin{proposition}\label{compo}
Let $\beta\colon\sM\to\sM\tens\sB$ and $\gamma\colon\sM\to\sM\tens\sC$ be quantum families of maps. If $\beta$ and $\gamma$ satisfy the Podle\'s condition then so does their composition $\beta\vt\gamma$.
\end{proposition}

\begin{proof}
We will write $[\mathcal{X}]$ for the closed linear span of a subset $\mathcal{X}$ of a normed vector space. We also write $[\,\cdots]$ instead of $[\{\,\cdots\}]$ whenever necessary.

Let $\mu_\sB$ be the multiplication map $\sB\atens\sB\to\sB$ and let $\sigma$ denote the flip $\sC\tens\sB\to\sB\tens\sC$. Assuming that $\beta$ and $\gamma$ satisfy the Podle\'s condition we compute
\[
\begin{split}
\bigl[(\beta\vt\gamma)(a)(\I\tens{X})\st{a}&\in\sM,\;X\in\sB\tens\sC\bigr]
=\bigl[(\beta\vt\gamma)(a)(\I\tens{b}\tens{c})\st{a}\in\sM,\;b\in\sB,\;c\in\sC\bigr]\\
&=\bigl[\bigl((\beta\tens\id)\gamma(a)\bigr)(\I\tens{b}\tens{c})\st{a}\in\sM,\;b\in\sB,\;c\in\sC\bigr]\\
&=\bigl[(\id\tens{\mu_\sB}\tens\id)(\beta\tens\sigma)\bigl((\gamma(a)(\I\tens{c}))\tens{b}\bigr)\st{a}\in\sM,\;b\in\sB,\;c\in\sC\bigr]\\
&=\bigl[(\id\tens{\mu_\sB}\tens\id)(\beta\tens\sigma)\bigl((a'\tens{c'})\tens{b}\bigr)\st{a'}\in\sM,\;b\in\sB,\;c'\in\sC\bigr]\\
&=\bigl[(\id\tens{\mu_\sB}\tens\id)(\beta\tens\id)\bigl((a'\tens{b})\tens{c'}\bigr)\st{a'}\in\sM,\;b\in\sB,\;c'\in\sC\bigr]\\
&=\bigl[\bigl(\beta(a')(\I\tens{b})\bigr)\tens{c'}\st{a'}\in\sM,\;b\in\sB,\;c'\in\sC\bigr]\\
&=\sM\tens\sB\tens\sC.
\end{split}
\]
(note that we used the fact that, as $\sM$ is finite dimensional, the ranges of $\beta$ and $\gamma$ contain only finite sums of simple tensors).
\end{proof}

From now on let us fix a faithful state $\phi$ on the \cst-algebra $\sM$. Let $\{e_1,\dotsc,e_n\}$ be a basis of $\sM$ which is orthonormal for the scalar product $\is{\cdot}{\cdot}_\phi$ defined by $\is{x}{y}_\phi=\phi(x^*y)$ for $x,y\in\sM$. Let $\{m^p_{k,l}\}_{k,l,p=1,\dotsc,n}$ be structure constants of $\sM$, i.e.\ complex numbers such that
\[
e_ke_l=\sum_{p=1}^nm_{k,l}^pe_p,\qquad{p}=1,\dotsc,n,
\]
and let $\{\lambda^i\}_{i=1,\dotsc,n}$ be the coefficients of $\I$ in the basis $\{e_1,\dotsc,e_n\}$:
\[
\sum_{i=1}^n\lambda^ie_i=\I.
\]
Finally let $\cT$ be the (invertible) scalar matrix with matrix elements $\{\tau_{k,l}\}_{k,l=1,\dotsc,n}$ such that
\[
e_l^*=\sum_{k=1}^n\tau_{k,l}e_k,\qquad{l}=1,\dotsc,n.
\]

Let $\sB$ be a unital \cst-algebra and $\beta\colon\sM\to\sM\tens\sB$ a linear map. Define a matrix
\[
\bb=\begin{bmatrix}
b_{1,1}&\dots&b_{1,n}\\
\vdots&\ddots&\vdots\\
b_{n,1}&\dots&b_{n,n}
\end{bmatrix}
\]
of elements of $\sB$ by
\[
\beta(e_j)=\sum_{i=1}^ne_i\tens{b_{i,j}},\qquad{i}=1,\dotsc,n.
\]
We will use the following notation:
\[
\Bar{\bb}=\begin{bmatrix}
b_{1,1}^*&\dots&b_{1,n}^*\\
\vdots&\ddots&\vdots\\
b_{n,1}^*&\dots&b_{n,n}^*
\end{bmatrix}.
\]

\begin{proposition}
Let $\sM$, $\phi$, $\{m^p_{k,l}\}_{k,l,p=1,\dotsc,n}$, $\{\lambda^i\}_{i=1,\dotsc,n}$ $\cT$ and  $\bb$  be as above. Then
\begin{enumerate}
\item\label{wang1} we have $(\phi\tens\id)\beta(x)=\phi(x)\I$ for all $x\in\sM$ if and only if
\begin{equation}\label{Wang1}
\sum_{i=1}^n\phi(e_i)b_{i,j}=\phi(e_j)\I,\qquad{j}=1,\dotsc,n,
\end{equation}
\item\label{wang2} $\beta$ is multiplicative if and only if
\begin{equation}\label{Wang2}
\sum_{k,l=1}^nm_{k,l}^pb_{k,i}b_{l,j}
=\sum_{q=1}^nm_{i,j}^qb_{p,q},\qquad{p},i,j=1,\dotsc,n,
\end{equation}
\item\label{wang3} $\beta(\I_{\sM})=\I_{\sM}\tens\I_{\sB}$ if and only if
\begin{equation}\label{Wang3}
\sum_{j=1}^n\lambda^jb_{i,j}=\lambda^i\I,\qquad{i}=1,\dotsc,n,
\end{equation}
\item\label{wang4} $\beta$ is a $*$-map if and only if
\begin{equation}\label{Wang4}
(\cT\tens\I)\Bar{\bb}=\bb(\cT\tens\I),
\end{equation}
\item\label{wang5} if \eqref{wang1}--\eqref{wang4} are satisfied then $\beta(\sM)(\I\tens\sB)$ is dense in $\sM\tens\sB$ if and only if $\bb$ is unitary.
\end{enumerate}
\end{proposition}

\begin{proof}
Points \eqref{wang1}, \eqref{wang2} and \eqref{wang3} are easy computations, while \eqref{wang4} and \eqref{wang5} are explained in \cite[Proof of Theorem 1.2]{apqs}.
\end{proof}

We are now ready to formulate the main result of this section.

\begin{theorem}\label{zeroth}
Fix a finite dimensional \cst-algebra $\sM$ with a faithful state $\phi$. Choose an orthonormal (with respect to the scalar product induced by $\phi$) basis $e_1, \ldots, e_n$ in $\sM$ and let $\{m^p_{k,l}\}_{k,l,p=1,\dotsc,n}$, $\{\lambda^i\}_{i=1,\dotsc,n}$ and $\cT$  be as above.
Let $\sA$ be the universal \cst-algebra generated by abstract elements $\{b_{i,j}\}_{i,j=1,\dotsc,n}$ such that the matrix
\[
\bb=\begin{bmatrix}
b_{1,1}&\dots&b_{1,n}\\
\vdots&\ddots&\vdots\\
b_{n,1}&\dots&b_{n,n}
\end{bmatrix}
\]
is unitary and relations \eqref{Wang1}--\eqref{Wang4} hold. Then the formula
\[
\alpha(e_j)=\sum_{i=1}^ne_i\tens{b_{i,j}},\qquad{j}=1,\dotsc,n.
\]
defines a unital $*$-homomorphism from $\sM$ to $\sM \tens \sA$ and for any $x\in\sM$ we have
\begin{equation}\label{invphi}
(\phi\tens\id)\alpha(x)=\phi(x)\I_\sA.
\end{equation}
Moreover
\begin{enumerate}
\item\label{first} for any unital \cst-algebra $\sC$ and any invertible quantum family of maps $\gamma\colon\sM\to\sM\tens\sC$ preserving the state $\phi$ there exists a unique unital $*$-homomorphism $\Lambda\colon\sA\to\sC$ such that the diagram
\[
\xymatrix{
\sM\ar[rr]^-{\alpha}\ar@{=}[d]&&\sM\tens\sA\ar[d]^{\id\tens\Lambda}\\
\sM\ar[rr]^-{\gamma}&&\sM\tens\sC}
\]
is commutative;
\item\label{second} there exists a unique unital $*$-homomorphism $\Delta\colon\sA\to\sA\tens\sA$ such that $(\alpha\tens\id)\comp\alpha=(\id\tens\Delta)\comp\alpha$. The map $\Delta$ is coassociative and the pair $(\sA,\Delta)$ gives rise to a compact quantum $\GG$, while $\alpha$ becomes an action of $\GG$ on $\sM$;
\item\label{third} if $\SSS$ is a compact quantum semigroup and $\rho\colon\sM\to\sM\tens\C(\SSS)$ is an action of $\SSS$ on $\sM$ preserving $\phi$ and satisfying Podle\'s condition then the unique $\Lambda\colon\sA\to\C(\SSS)$ such that $\rho=(\id\tens\Lambda)\comp\alpha$ satisfies
\[
\Delta_{\SSS}\comp\Lambda=(\Lambda\tens\Lambda)\comp\Delta;
\]
\item\label{fourth} $\GG$ is isomorphic to Wang's quantum automorphism group of $(\sM,\phi)$. 
\end{enumerate}
\end{theorem}

\begin{proof}
The definition of $\sA$ is such that the linear mapping
\[
e_j\longmapsto\sum_{i=1}^ne_i\tens{b_{i,j}},\qquad{j}=1,\dotsc,n
\]
extends to a unital $*$-homomorphism $\sM\to\sM\tens\sA$ such that for any $x\in\sA$ we have \eqref{invphi}.

Ad \eqref{first}. If $\gamma\colon\sM\to\sM\tens\sC$ is a quantum family of invertible maps preserving $\phi$, then defining $\{c_{i,j}\}_{i,j=1,\dotsc,n}$ by
\[
\gamma(e_j)=\sum_{i=1}^ne_i\tens{c_{i,j}},\qquad{j}=1,\dotsc,n,
\]
we obtain elements satisfying relations we used to define the \cst-algebra $\sA$. Therefore there exists a unique unital $*$-homomorphism $\Lambda\colon\sA\to\sC$ such that
\[
\Lambda(b_{i,j})=c_{i,j},\qquad{i},j=1,\dotsc,n
\]
and it is immediate that $(\id\tens\Lambda)\comp\alpha=\gamma$. This last property defines $\Lambda$ uniquely because $\sA$ is generated by $\{b_{i,j}\}_{i,j=1,\dotsc,n}$.

Ad \eqref{second}. We will proceed as in proofs of \cite[Theorems 6, 16 \& 21]{qs}. By its very definition $\alpha\colon\sM\to\sM\tens\sA$ is a quantum family of invertible maps preserving $\phi$. Now Proposition \ref{compo} asserts that so is $\alpha\vt\alpha$. Thus statement \eqref{first} gives a unique $\Delta\colon\sA\to\sA\tens\sA$ such that
\[
\alpha\vt\alpha=(\id\tens\Delta)\comp\alpha.
\]
Coassociativity of $\Delta$ follows from the associativity of the operation of composition of quantum families of maps exactly as in e.g.~\cite[Proof of Theorem 6(2)]{qs} (we need here the obvious fact that $\sA$ is generated by $\bigl\{(\omega\tens\id)\alpha(m)\st{m}\in\sM;\;\omega\in\sM^*\bigr\}$). It is also easy to check that
\[
\Delta(b_{i,j})=\sum_{k=1}^nb_{i,k}\tens{b_{k,j}},\qquad{i},j=1,\dotsc,n
\]
(which might as well be used to prove coassociativity of $\Delta$).

As the matrix $\bb$ is unitary and its transpose $\bb^\top=\bigl(\Bar{\bb}\bigr)^*$ is invertible (due to relation \eqref{Wang4}), the results of \cite{RemCQG} guarantee that $\Delta$ defines on $\sA$ the structure of the algebra of continuous functions on a compact quantum group. The remaining statements of \eqref{second} are clear.

Ad \eqref{third}. Let $\widetilde{\alpha}\colon\sM\to\sM\tens\widetilde{\sA}$ be the quantum family of all maps on $\sM$ preserving $\phi$. Then there exists a unique unital $*$-homomorphism $\widetilde{\Lambda}\colon\widetilde{\sA}\to\sA$ such that
\[
\alpha=(\id\tens\widetilde{\Lambda})\comp\widetilde{\alpha}.
\]
Moreover, since $\sA$ is generated by $\bigl\{(\omega\tens\id)\alpha(m)\st{m}\in\sM;\;\omega\in\sM^*\bigr\}$, one easily sees that $\widetilde{\Lambda}$ is surjective, as for each $m\in\sM$ and $\omega\in\sM^*$ we have
\[
(\omega\tens\id)\alpha(m)=\widetilde{\Lambda}\bigl((\omega\tens\id)\widetilde{\alpha}(m)\bigr).
\]
By \cite[Theorem 16]{qs} there is a comultiplication $\widetilde{\Delta}$ on $\widetilde{\sA}$ and we have
\[
\Delta\comp\widetilde{\Lambda}=(\widetilde{\Lambda}\tens\widetilde{\Lambda})\comp\widetilde{\Delta}.
\]

Now, if $\rho\colon\sM\to\sM\tens\C(\SSS)$ is an action of a compact quantum semigroup $\SSS$ preserving $\phi$ then there exists a unique $\Theta\colon\widetilde{\sA}\to\C(\SSS)$ such that
\[
\rho=(\id\tens\Theta)\comp\widetilde{\alpha}.
\]
Moreover, we have $\Delta_\SSS\comp\Theta=(\Theta\tens\Theta)\comp\widetilde{\Delta}$.

Finally, by \eqref{first}, there exists a unique $\Lambda\colon\sA\to\C(\SSS)$ such that
\[
\rho=(\id\tens\Lambda)\comp\alpha.
\]
By uniqueness of $\Theta$, we obviously have
\[
\Theta=\Lambda\comp\widetilde{\Lambda}.
\]

All this information is summarized in the following commutative diagram:
\begin{equation}\label{diagr}
\xymatrix
{
\widetilde{\sA}\ar[rr]^{\widetilde{\Delta}}\ar[d]^{\widetilde{\Lambda}}\ar@<-1ex>@/^-5ex/[dd]_{\Theta}
&&\widetilde{\sA}\tens\widetilde{\sA}\ar[d]_{\widetilde{\Lambda}\tens\widetilde{\Lambda}}\ar@<3ex>@/^5ex/[dd]^{\Theta\tens\Theta}\\
\sA\ar[rr]^{\Delta}\ar[d]^{\Lambda}&&\sA\tens\sA\\
\C(\SSS)\ar[rr]^{\Delta_\SSS}&&\C(\SSS)\tens\C(\SSS)
}
\end{equation}

It remains to prove that we can complete \eqref{diagr} with the map $\Lambda\tens\Lambda\colon\sA\tens\sA\to\C(\SSS)\tens\C(\SSS)$. But this is not difficult: for $a\in\sA$ we let $\widetilde{a}\in\widetilde{\sA}$ be its lift through $\widetilde{\Lambda}$. Clearly
\[
\begin{split}
\Delta_\SSS\bigl(\Lambda(a)\bigr)&=\Delta_\SSS\bigl((\Lambda\comp\widetilde{\Lambda})(\widetilde{a})\bigr)\\
&=\Delta_\SSS\bigl(\Theta(\widetilde{a})\bigr)\\
&=(\Theta\tens\Theta)\bigl(\widetilde{\Delta}(\widetilde{a})\bigr)\\
&=\bigl([\Lambda\comp\widetilde{\Lambda}]\tens[\Lambda\comp\widetilde{\Lambda}]\bigr)\bigl(\widetilde{\Delta}(\widetilde{a})\bigr)\\
&=\bigl((\Lambda\tens\Lambda)\comp(\widetilde{\Lambda}\tens\widetilde{\Lambda})\bigr)\bigl(\widetilde{\Delta}(\widetilde{a})\bigr)\\
&=(\Lambda\tens\Lambda)\bigl((\widetilde{\Lambda}\tens\widetilde{\Lambda})\bigl(\widetilde{\Delta}(\widetilde{a})\bigr)\bigr)\\
&=(\Lambda\tens\Lambda)\bigl(\Delta\bigl(\widetilde{\Lambda}(\widetilde{a})\bigr)\bigr)\\
&=(\Lambda\tens\Lambda)\bigl(\Delta(a)\bigr).
\end{split}
\]

Ad \eqref{fourth}. Clearly $\GG$ defined by $\C(\GG)=\sA$ and $\Delta_\GG=\Delta$ is a compact quantum group possessing the universal property required of the quantum automorphism group of $(\sM,\phi)$. This is simply \eqref{first} applied only to quantum families of the form $\gamma\colon\sM\to\sM\tens\C(\HH)$, where $\HH$ is a compact quantum group and $\gamma$ is its action on $\sM$ preserving $\phi$.
\end{proof}

\begin{remark}\label{RemInv}
\noindent\begin{enumerate}
\item The theorem implies that in fact $\GG$ depends only on the choice of $\sM$ and $\phi$, and not on the choice of the basis $(e_1, \ldots, e_m)$. In fact $\GG$ is characterised up to an isomorphism as the \emph{compact quantum semigroup} enjoying the universal property in (3); similarly one can say that $(\sA, \alpha)$ is  characterised up to an isomorphism as the \emph{quantum family of maps on $\sM$} enjoying the universal property in (2).
\item It is important to note that the \cst-algebra $\C(\GG)$ is generated by
\[
\bigl\{(\omega\tens\id)\alpha(m)\st{m}\in\sM,\;\omega\in\sM^*\bigr\}
\]
(in other words the action $\alpha$ is \emph{faithful}). We will use this fact repeatedly in what follows.
\item\label{RemInv2} As mentioned in the beginning of this section, Theorem \ref{zeroth} provides an alternative description of Wang's quantum automorphism group of $(\sM,\phi)$ which we will from now on denote by the symbol $\Aut(\sM,\phi)$. The crucial point is that these objects possess a more general universal property than in their original definition. In \cite{Wang} the only requirement was that $\GG$ together with its action on $\sM$ be universal for actions of compact quantum groups preserving $\phi$, but it turns out that it is universal for all quantum families of invertible maps preserving $\phi$.
\item Theorem \ref{genbeta} provides also a more conceptual interpretation of the results of \cite{apqs}. The actions of quantum semigroups studied in that paper are nothing other than certain quantum families of invertible maps.
\item One can easily see that $\Aut(\sM,\phi)$ has a formally even stronger universal property, namely, for any \cst-algebra $\sB$ (not necessarily unital) and any quantum family $\beta\in\Mor(\sM,\sM\tens\sB)$ preserving $\phi$ and satisfying Podle\'s condition (so, in particular, elements of the form $\beta(m)(\I\tens{x})$ must \emph{belong} to $\sM\tens\sB$) there exists a unique $\Lambda\in\Mor\bigl(\C(\Aut(\sM,\phi)),\sB\bigr)$ such that
\[
\beta=(\id\tens\Lambda)\comp\alpha.
\]
This follows from the fact that a morphism from a unital \cst-algebra to $\sB$ is a unital $*$-homomorphism into $\M(\sB)$.
\end{enumerate}
\end{remark}

\section{Hopf images in the compact quantum group context}

In this section we discuss a purely analytic approach to Hopf images of Banica and Bichon (\cite{BanicaBichon}) in the compact quantum group context. Recall that if $\HH$, $\GG$ are compact quantum groups then we say that $\HH$ is a \emph{(closed) quantum subgroup} of $\GG$ if there exists a surjective morphism $\pi\in \Mor (\C(\GG), \C(\HH))$ intertwining the respective coproducts. For an exhaustive discussion of the notion of a closed quantum subgroup we refer to the article \cite{dkss}.


\begin{theorem}\label{HopfImage}
Let $\GG$ be a compact quantum group, $\sB$ a unital \cst-algebra and let $\Lambda\colon\C(\GG)\to\sB$ be a unital $*$-homomorphism. Define
\[
\Lambda_n=(\underbrace{\Lambda\tens\dotsm\tens\Lambda}_n)\comp\Delta^{(n-1)}\colon\C(\GG)\longrightarrow\sB^{\tens{n}}
\]
and let $\sJ=\bigcap\limits_{n=1}^\infty\ker\Lambda_n$. Let $\sS=\C(\GG)/\sJ$ and let $\pi\colon\C(\GG)\to\sS$ be the quotient map. Then
\begin{enumerate}
\item there exists a unique $\cDelta\colon\sS\to\sS\tens\sS$ such that $\cDelta\comp\pi=(\pi\tens\pi)\comp\Delta_\GG$.
\item The pair $(\sS,\cDelta)$ gives rise to a compact quantum group $\KK_\Lambda$ -- more precisely, $\sS= \C(\KK_\Lambda)$ and $\cDelta$ is the corresponding coproduct.
\item The map $\Lambda$ factorizes uniquely as
\begin{equation}\label{Ltp}
\Lambda=\theta\comp\pi
\end{equation}
for a certain unital $*$- homomorphism $\theta\colon\sS\to\sB$.
\item If $\HH$ is a compact quantum subgroup of $\GG$ with corresponding surjective quantum group morphism $\tau\colon\C(\GG)\to\C(\HH)$ such that $\Lambda$ factorizes as
\[
\Lambda=\chi\comp\tau
\]
for some $\chi\colon\C(\HH)\to\sB$ then there exists a unique surjective map $\rho\colon\C(\HH)\to\sS$ such that $\pi=\rho\comp\tau$. This $\rho$ is a compact quantum group morphism.
\end{enumerate}
\end{theorem}

\begin{proof}
First note that for any $n,m\in\NN$ we have
\[
\begin{split}
(\Lambda_n\tens\Lambda_m)\comp\Delta&=\bigl((\underbrace{\Lambda\tens\dotsm\tens\Lambda}_n)\tens(\underbrace{\Lambda\tens\dotsm\tens\Lambda}_m)\bigr)
\comp(\Delta^{(n-1)}\tens\Delta^{(m-1)})\comp\Delta\\
&=(\underbrace{\Lambda\tens\dotsm\tens\Lambda}_{n+m})\comp\Delta^{(n+m-1)}=\Lambda_{n+m}.
\end{split}
\]

Now let us take $x\in\sJ$. We will show that
\begin{equation}\label{zero}
(\pi\tens\pi)\Delta_\GG(x)=0.
\end{equation}
To that end, let $v=(\pi\tens\id)\Delta_\GG(x)\in\sS\tens\C(\GG)$. In order to have \eqref{zero} it is enough to show that $(\id\tens\pi)(v)=0$.

Assume to the contrary that $(\id\tens\pi)(v)\neq{0}$. Then there exists a functional $\omega\in\sS^*$ such that
\[
(\omega\tens\id)\bigl((\id\tens\pi)(v)\bigr)\neq{0}.
\]
Or, in other words,
\begin{equation}\label{nonzero}
\pi\bigl((\omega\tens\id)(v)\bigr)\neq{0}
\end{equation}
(cf.~\cite[Section 1.5.4(b)]{SW}). 

Since $x\in\sJ$, for any $n,m\in\NN$
\begin{equation}\label{Jkl}
(\Lambda_n\tens\Lambda_m)\Delta(x)=\Lambda_{n+m}(x)=0.
\end{equation}
Fix $m$ and $\nu\in(\underbrace{\sB\tens\dotsm\tens\sB}_m)^*$ and let
\[
z_m=\bigl(\id\tens[\nu\comp\Lambda_m]\bigr)\Delta(x).
\]
For any $n\in\NN$ we have by \eqref{Jkl}
\[
\Lambda_n(z_m)=(\id\tens\nu)(\Lambda_n\tens\Lambda_m)\Delta(x)=0,
\]
which means that $z_m\in\sJ$, i.e.~$\pi(z_m)=0$. Thus
\[
\begin{split}
0=\omega\bigl(\pi(z_m)\bigr)&=\bigl(\omega\tens[\nu\comp\Lambda_m]\bigr)\Delta(x)\\
&=(\omega\tens\nu)(\id\tens\Lambda_m)\Delta(x)\\
&=\nu\bigl(\Lambda_m((\omega\tens\id)\Delta(x))\bigr).
\end{split}
\]
Since this is true for all $\nu$, we get
\[
\Lambda_m\bigl((\omega\tens\id)\Delta(x)\bigr)=0,\qquad{m}\in\NN.
\]
But his means that $\pi\bigl((\omega\tens\id)\Delta(x)\bigr)=0$ which contradicts \eqref{nonzero}. This proves \eqref{zero}.

It is now easy to see that there exists a unique comultiplication $\cDelta\colon\sS\to\sS\tens\sS$ such that
\[
\cDelta\comp\pi=(\pi\tens\pi)\comp\Delta_\GG.
\]
Clearly $\cDelta$ is coassociative (simply apply $(\pi\tens\pi\tens\pi)$ to both sides of $(\Delta_\GG\tens\id)\comp\Delta_\GG=(\id\tens\Delta_\GG)\comp\Delta_\GG$). Moreover, since $\pi$ is surjective and
\[
\begin{array}{r@{\;=\;}l}
(\pi\tens\pi)\bigl(\Delta(a)(\I\tens{b})\bigr)&\cDelta\bigl(\pi(a)\bigr)\bigl(\I\tens\pi(b)\bigr),\\
(\pi\tens\pi)\bigl((a\tens\I)\Delta(b)\bigr)&\bigl(\I\tens\pi(a)\bigr)\cDelta\bigl(\pi(b)\bigr),
\end{array}
\qquad{a},b\in\sA,
\]
the density conditions of \cite[Definition 2.1]{cqg} are satisfied and $\cDelta$ defines on $\sS$ the structure of an algebra of functions on a compact quantum group.

Observe now that since $\ker\Lambda=\ker\Lambda_1\subset\sJ$, there exists a unique $\theta\colon\sS\to\sB$ such that \eqref{Ltp} holds.

Now let $\HH$ be a compact quantum subgroup of $\GG$ with corresponding surjection $\tau\colon\C(\GG)\to\C(\HH)$ such that
\[
\Lambda=\chi\comp\tau
\]
for some $\chi\colon\C(\HH)\to\sB$.

Take $y\in\ker\tau$. Then for any $n \in \NN$ we have
\[
\begin{split}
\Lambda_n(y)=(\underbrace{\Lambda\tens\dotsm\tens\Lambda}_n)\Delta_\GG^{(n-1)}(y)
&=(\underbrace{\chi\tens\dotsm\tens\chi}_n)(\underbrace{\tau\tens\dotsm\tens\tau}_n)\Delta_\GG^{(n-1)}(y)\\
&=(\underbrace{\chi\tens\dotsm\tens\chi}_n)\Delta_\HH^{(n-1)}\bigl(\tau(y)\bigr)=0
\end{split}
\]
because $\tau$ is a compact quantum group morphism. It follows that for each $n \in \NN$ we have $\ker\tau\subset\ker\Lambda_n$, so $\ker\tau\subset\sJ$. Therefore there is a unique surjective $\rho\colon\C(\HH)\to\sS$ such that $\pi=\rho\comp\tau$. The map $\rho$ is a compact quantum group morphism, as for any $z\in\C(\HH)$ we can find a $z'\in\C(\GG)$ such that $z=\tau(z')$ and thus
\[
\begin{split}
(\rho\tens\rho)\Delta_\HH(z)&=(\rho\tens\rho)\Delta_\HH\bigl(\tau(z')\bigr)\\
&=(\rho\tens\rho)(\tau\tens\tau)\Delta_\GG(z')\\
&=(\pi\tens\pi)\Delta_\GG(z')\\
&=\cDelta\bigl(\pi(z')\bigr)\\
&=\cDelta\bigl(\rho(\tau(z'))\bigr)=\cDelta\bigl(\rho(z)\bigr)
\end{split}
\]
so that $(\rho\tens\rho)\comp\Delta_\HH=\cDelta\comp\rho$.
\end{proof}

Remark here that apart from the fact that we remain here all the time on the level of \cst-algebras, there is one other notable difference with the general Hopf algebra setup, observed already in \cite{HopfIdempotents}: when constructing $\sJ$ we only need to take care of the iterated `convolutions' of the map $\Lambda$, so that, as opposed to the situation in \cite{BanicaBichon}, the antipode does not play any role. In the next corollary we note that the same construction works if we start from a morphism taking values into a not-necessarily unital $\cst$-algebra.

\begin{corollary}\label{CorHopfImage}
\sloppy
Let $\GG$ be a compact quantum group, $\sB$ a (not necessarily unital) \cst-algebra and let $\Lambda\in\Mor\bigl(\C(\GG),\sB\bigr)$. Then there exists a unique \cst-algebra $\sS$ equipped with a surjective unital $*$-homomorphism $\pi\colon\C(\GG)\to\sS$ such that
\begin{enumerate}
\item there exists a $\cDelta\colon\sS\to\sS\tens\sS$ such that $\cDelta\comp\pi=(\pi\tens\pi)\comp\Delta_\GG$.
\item The pair $(\sS,\cDelta)$ gives rise to a compact quantum group $\KK_\Lambda$.
\item The map $\Lambda$ factorizes uniquely as $\Lambda=\theta\comp\pi$ for a certain $\theta\in\Mor(\sS,\sB)$.
\item If $\HH$ is a compact quantum subgroup of $\GG$ with corresponding surjective quantum group morphism $\tau\colon\C(\GG)\to\C(\HH)$ such that $\Lambda$ factorizes as $\Lambda=\chi\comp\tau$ for some $\chi\in\Mor\bigl(\C(\HH),\sB\bigr)$ then there exists a unique surjective map $\rho\colon\C(\HH)\to\sS$ such that $\pi=\rho\comp\tau$. This $\rho$ is a compact quantum group morphism.
\end{enumerate}
\end{corollary}

\begin{proof}
Replace $\sB$ by $\M(\sB)$ and use Theorem \ref{HopfImage}.
\end{proof}

\begin{definition}\label{DefHI}
Let $\GG$ be a compact quantum group, $\sB$ a \cst-algebra and $\Lambda\in\Mor\bigl(\C(\GG),\sB\bigr)$. The compact quantum group  $\KK_\Lambda$ constructed from this data in Corollary \ref{CorHopfImage} is called the \emph{Hopf image} of $\Lambda$. 
\end{definition}

\begin{remark}\label{RemHI}
Let $\GG$, $\sB$ and $\Lambda$ be as in Corollary \ref{CorHopfImage}.
\begin{enumerate}
\item\label{RemHI1} If $\sB$ is commutative then $\KK_\Lambda$ is a classical compact group. Indeed, commutativity of $\sB$ implies that the commutator ideal of $\C(\GG)$ is contained in $\ker(\underbrace{\Lambda\tens\dotsm\tens\Lambda}_n)\comp\Delta_\GG^{(n-1)}$ for all $n \in \NN$. It follows that $\C(\KK_\Lambda)$ is commutative.
\item\label{RemHI2} If $\sB$ is equipped with a comutiplication $\Delta_\sB$ such that $(\sB,\Delta_\sB)$ is the algebra of continuous functions on a compact quantum group $\HH$ and $\Lambda$ is a surjective compact quantum group morphism then $\KK_\Lambda$ is isomorphic to $\HH$. This easily follows from the fact that in this situation we have
\[
\ker\Lambda_n=\ker\Lambda
\]
for all $n\in\NN$.
\end{enumerate}
\end{remark}

Theorem \ref{HopfImage} was recently applied in \cite{Simeng} to show the following result: suppose $\GG$ is a compact quantum group, $\sA$ is a unital $*$-algebra with a faithful state $\phi$ and let $\Lambda:\Pol(\GG)\to \sA$ be a unital $*$-algebra homomorphism. Define $\omega = \phi \circ \Lambda$ and put $\tilde{\omega} = \lim_{n \to \infty} \frac{1}{n} \sum_{k=1}^n \omega^{\star k}$. Then $\pi$ is \emph{inner faithful} (i.e.\ the Hopf image of $\Lambda$ is $\GG$) if and only if $\tilde{\omega}= h_{\GG}$.

\section{Quantum subgroup generated by a family of quantum subgroups}
In this section we use the results of the last one to provide a construction of the quantum subgroup generated by a family of quantum subgroups. Note that very recently an alternative (but equivalent) approach to this concept was introduced, and very successfully applied, by Brannan, Collins and Vergnioux in \cite{BCV}.

Let $\{\HH_i\}_{i\in{I}}$ be a family of closed quantum subgroups of a compact quantum group $\GG$ with corresponding surjections
\[
\pi_i\colon\C(\GG)\longrightarrow\C(\HH_i),\qquad{i}\in{I}.
\]
Denote by $\sB$ the direct sum $\bigoplus\limits_{i\in{I}}\C(\HH_i)$ and let $\Lambda\colon\C(\GG)\to\M(\sB)$ by
\[
\Lambda(x)=\bigoplus_{i\in{I}}\pi_i(x)
\]
(clearly $\Lambda\in\Mor\bigl(\C(\GG),\sB\bigr)$). As declared in Definition \ref{DefHI} we denote by $\KK_\Lambda$ the compact quantum group obtained through the Hopf image construction from $\Lambda$. We have the surjective compact quantum group morphism $\pi\colon\C(\GG)\to\C(\KK_\Lambda)$ and $\theta\in\Mor\bigl(\C(\KK_\Lambda),\sB\bigr)$ such that
\[
\Lambda=\theta\comp\pi.
\]
For each $i$ let $\mathrm{p}_i$ be the canonical projection $\sB\to\C(\HH_i)$ and let $\theta_i=\mathrm{p}_i\comp\theta\colon\C(\KK_\Lambda)\to\C(\HH_i)$. Then $\theta_i$ identifies $\HH_i$ as a subgroup of $\KK_\Lambda$: for any $z\in\C(\KK_\Lambda)$ there is a $z'\in\C(\GG)$ such that $z=\pi(z')$. Therefore
\[
\begin{split}
(\theta_i\tens\theta_i)\Delta_{\KK_\Lambda}(z)&=(\theta_i\tens\theta_i)(\pi\tens\pi)\Delta_\GG(z')\\
&=(\mathrm{p}_i\tens\mathrm{p}_i)\Lambda\tens\Lambda)\Delta_\GG(z')\\
&=(\pi_i\tens\pi_i)\Delta_\GG(z')=\Delta_{\HH_i}(z).
\end{split}
\]
It is also easy to see that each $\theta_i$ is surjective (although $\theta$ is not) and $\pi_i$ factorizes as $\pi_i=\theta_i\comp\pi$.

\sloppy
Moreover $\KK_\Lambda$ can be described as the smallest (closed) subgroup of $\GG$ which contains the subgroups $\{\HH_i\}_{i\in{I}}$ in the sense that if $\HH$ is a quantum subgroup of $\GG$ with corresponding morphism $\tau\colon\C(\GG)\to\C(\HH)$ such that for each $i$ the map $\pi_i$ factorizes as
\[
\pi_i=\chi_i\comp\tau
\]
for some $\chi_i\colon\C(\HH)\to\C(\HH_i)$ then $\KK_\Lambda$ is a subgroup of $\HH$. This follows from the fact that putting
\[
\chi\colon\C(\HH)\ni{x}\longmapsto\bigoplus_{i\in{I}}\chi_i(x)\in\M(\sB)
\]
we define a morphism from $\C(\HH)$ to $\sB$ which obviously satisfies $\Lambda=\chi\comp\tau$. Thus we can use the universal property of the Hopf image. Note that this description guarantees indeed that if $\GG$ is a classical compact group with subgroups $H_i$, then $\KK_{\Lambda}$ constructed above is indeed (isomorphic to) the smallest closed subgroup generated by $H_i$.

\begin{definition}
The closed quantum subgroup $\KK_\Lambda$ obtained above from the family of quantum subgroups $\{\HH_i\}_{i\in{I}}$ of $\GG$ is called the quantum subgroup of $\GG$ \emph{generated by the subgroups $\{\HH_i\}_{i\in{I}}$}.
\end{definition}

\begin{remark}
\noindent
\begin{enumerate}
\item If the family $\{\HH_i\}_{i\in{I}}$ is trivial in the sense that all $\HH_i$ are equal (this means that also all $\pi_i$ are the same map) then the quantum subgroup generated by $\{\HH_i\}_{i\in{I}}$ coincides with $\HH_{i_0}$ for any $i_0\in{I}$. This is a clear consequence of its universal property.

Similarly one can show that if there are repetitions in the collection $\{\HH_i\}_{i\in{I}}$ (and $\{\pi_i\}_{i\in{I}}$) then they can be appropriately removed from the list.
\item Let $\{\HH_i\}_{i\in{I}}$ and $\{\HH_j\}_{j\in{J}}$ be two families of closed subgroups of $\GG$ and let
\[
\begin{split}
\Lambda&\in\Mor\bigl(\C(\GG),\sB_I\bigr)\\
\Gamma&\in\Mor\bigl(\C(\GG),\sB_J\bigr),
\end{split}
\]
where
\[
\sB_I=\bigoplus_{i\in{I}}\C(\HH_i)\quad\text{and}\quad\sB_J=\bigoplus_{j\in{J}}\C(\HH_j),
\]
be corresponding morphisms (as above). Denote by $\KK_I$ and $\KK_J$ the quantum subgroups of $\GG$ generated by $\{\HH_i\}_{i\in{I}}$ and $\{\HH_j\}_{j\in{J}}$ respectively. Then the quantum subgroup of $\GG$ generated by $\{\HH_i\}_{i\in{I}}\cup\{\HH_j\}_{j\in{J}}$ is canonically isomorphic to the quantum subgroup of $\GG$ generated by $\{\KK_I,\KK_J\}$. This is, again, an easy consequence of the universal property of the quantum subgroup generated by a given family.

Moreover, one can easily generalize this fact to not necessarily finite collections of families of subgroups of $\GG$.
\item Let $\Gamma$ be a discrete group and suppose that $\GG = \widehat{\Gamma}$, so that $\C(\GG) = \cst(\Gamma)$. It is then well-known (see for example \cite{dkss}) that subgroups of $\GG$ correspond to \emph{normal} subgroups of $\Gamma$, in the following sense: if $\HH$ is a quantum subgroup of $\GG$, then there exists $S \triangleleft \Gamma$ so that $\HH = \widehat{\Gamma/S}$, and the associated morphism $\pi \in \Mor(\cst(\Gamma), \cst (\Gamma/S))$ is induced by the quotient map $\Gamma \mapsto \Gamma/S$. Further given a family of quantum subgroups $\{\HH_i\}_{i\in{I}}$ of $\widehat{\Gamma}$ (with corresponding normal subgroups $S_i \triangleleft \Gamma$) we can easily check, for example using the universal properties, that the quantum subgroup generated by $\{\HH_i\}_{i\in{I}}$ is $\widehat{\Gamma/S}$, where $S= \bigcap_{i \in I} S_i$.
\item If $\GG$ is a compact quantum group, and $\{\HH_i\}_{i\in{I}}$ is a collection of its \emph{classical} subgroups (i.e.\ each $\HH_i$ is a quantum subgroup of $\GG$ and $\C(\HH_i)$ is commutative), then the quantum subgroup generated by $\{\HH_i\}_{i\in{I}}$ is also classical. This follows from Remark \ref{RemHI1}. The quantum subgroup generated by duals of classical groups need not be a dual of a classical group (as already classically the subgroup generated by two abelian groups need not be abelian).

\end{enumerate}
\end{remark}

We finish this section by sketching an alternative, dual description of the quantum subgroup generated by $\{\HH_i\}_{i\in{I}}$. Recall that if $\HH$ is a quantum subgroup of $\GG$, and $\pi\colon\C(\GG)\to\C(\HH)$ is the corresponding surjective morphism, then there exists a unique \emph{dual} morphism $\widehat{\pi}\colon\ell^{\infty}(\widehat{\HH})\to\ell^{\infty}(\widehat{\GG})$, where $\ell^{\infty}(\widehat{\HH})$ and $\ell^{\infty}(\widehat{\GG})$ denote respectively the von Neumann algebras of bounded functions on the dual, discrete quantum groups (see \cite{dkss} for the details).

\begin{theorem} \label{generatedsubgroupdual}
Let $\GG$ be a compact quantum group with the family of quantum subgroups $\{\HH_i\}_{i\in{I}}$. Consider the dual morphisms $\widehat{\pi}_i\colon\ell^{\infty} (\widehat{\HH_i})\to\ell^{\infty}(\widehat{\GG})$. The smallest von Neumann algebra invariant under the coproduct generated by $\widehat{\pi}_i\bigl(\ell^{\infty}(\widehat{\HH_i})\bigr) $ inside $\ell^{\infty}(\widehat{\GG})$ is isomorphic to $\widehat{\pi}\bigl(\ell^{\infty}(\widehat{\KK})\bigr)$, where $\KK$ is the quantum subgroup generated by $\{\HH_i\}_{i\in{I}}$ and $\pi\colon\C(\GG)\to\C(\KK)$ is the morphism identifying $\KK$ as a quantum subgroup of $\GG$.
\end{theorem}

\begin{proof}

Denote the von Neumann algebra introduced in the formulation of the theorem by $\sM$. It follows from \cite[Theorem 3.1]{NeY} that $\sM=\widehat{\pi}\bigl(\ell^{\infty}(\widehat{\KK})\bigr)$, where $\KK$ is a quantum subgroup of $\GG$ and $\pi\colon\C(\GG)\to\C(\KK)$ is the respective quantum group morphism. It now suffices to observe that the construction of the dual morphism is functorial, by which we mean that if for example $\HH_1\subset \HH_2 \subset \HH_3$ are inclusions of compact quantum groups, then we have  $\widehat{\pi}_{1,3} =\widehat{\pi}_{2,3} \circ \widehat{\pi}_{1,2}$, and use the universal properties.
\end{proof}

\section{Quantum groups generated by quantum families of invertible maps}\label{Sgen}

Throughout this section we let $(\sM,\phi)$ be a finite dimensional \cst-algebra with a faithful state. Wang's quantum automorphism group $\Aut(\sM,\phi)$ of this pair will be denoted by $\GG$, and its action on $\sM$ by $\alpha$. We show how we can use the idea of a Hopf image to define the quantum group generated by a given family of invertible maps.

\begin{theorem}\label{genbeta}
Let $\sB$ be a \cst-algebra and let $\beta\in\Mor(\sM,\sM\tens\sB)$ be a quantum family of invertible maps preserving $\phi$. Then there exists a compact quantum group $\KK_\beta$ equipped with a $\phi$-preserving action on $\sM$, $\Bar{\beta}\colon\sM\to\sM\tens\C(\KK_\beta)$ and a map $\theta\in\Mor\bigl(\C(\KK_\beta),\sB\bigr)$ such that
\begin{equation}\label{bbb}
\beta=(\id\tens\theta)\comp\Bar{\beta}
\end{equation}
determined uniquely by the property that if $\HH$ is a compact quantum group with a faithful $\phi$-preserving action $\gamma\colon\sM\to\sM\tens\C(\HH)$ and a map $\chi\in\Mor\bigl(\C(\HH),\sB\bigr)$ such that
\[
\beta=(\id\tens\chi)\comp\gamma
\]
then there exists a unique surjective map $\rho\colon\C(\HH)\to\C(\KK_\beta)$ such that $\chi=\theta\comp\rho$.
\end{theorem}

\begin{proof}
Since $\beta$ preserves $\phi$ and satisfies Podle\'s condition, there exists a unique $\Lambda\in\Mor\bigl(\C(\GG),\sB)$ such that
\[
\beta=(\id\tens\Lambda)\comp\alpha.
\]
Let $\KK_\beta$ be the Hopf image associated with $\Lambda$. In particular we have a surjective map $\pi\colon\C(\GG)\to\C(\KK_\beta)$ intertwining the respective coproducts and $\theta\in\Mor\bigl(\C(\KK_\beta),\sB\bigr)$ such that
\[
\Lambda=\theta\comp\pi.
\]
Clearly $\Bar{\beta}=(\id\tens\pi)\comp\alpha$ is an action of $\KK_\beta$ on $\sM$, $\phi$ is invariant for $\Bar{\beta}$ and \eqref{bbb} holds.

Let us check the universal property of $\KK_\beta$: let $\HH$ and $\gamma$ be as in the formulation of the theorem. Since $\phi$ is invariant for $\gamma$, there exists a unique $\tau\colon\C(\GG)\to\C(\HH)$ such that
\[
\gamma=(\id\tens\tau)\comp\alpha
\]
and $\tau$ is surjective due to faithfulness of $\gamma$. Now assume, as in the formulation, that we also have a $\chi\in\Mor\bigl(\C(\HH),\sB\bigr)$ such that
\[
\beta=(\id\tens\chi)\comp\gamma.
\]
Then obviously $\chi\comp\tau=\Lambda$ due to the universal property of $\GG$ (the map $\Lambda$ is unique such that $(\id\tens\Lambda)\comp\alpha=\beta$). Therefore, the universal property of $\KK_\beta$ shows that there is a unique surjective $\rho\colon\C(\HH)\to\C(\KK_\beta)$ such that $\chi=\theta\comp\rho$.
\end{proof}

\begin{definition}
Let $\beta\in\Mor(\sM,\sM\tens\sB)$ be a quantum family of invertible maps preserving $\phi$. The compact quantum group $\KK_\beta$ constructed in theorem \ref{genbeta} will be called the compact quantum group \emph{generated} by the family $\beta$.
\end{definition}

\begin{remark}\label{Rem:genInvertible}
Let $\beta\in\Mor(\sM,\sM\tens\sB)$ be a quantum family of invertible maps preserving $\phi$.
\begin{enumerate}
\item The compact quantum group generated by $\beta$ can be interpreted as obtained by taking all possible compositions of members of the family $\beta$. Indeed, the maps $\{\Lambda_n\}_{n\in\NN}$ used in the proof of Theorem \ref{HopfImage} to construct the Hopf image (i.e.~$\KK_\beta$ in this case) are the unique elements of $\Mor\bigl(\C(\GG),\sB^{\tens{n}}\bigr)$ such that
\[
\underbrace{\beta\vt\dotsm\vt\beta}_n=(\id\tens\Lambda_n)\comp\alpha.
\]
Indeed,
\[
\begin{split}
\beta^{\vartriangle{n}}&=\bigl((\id\tens\Lambda)\comp\alpha\bigr)^{\vartriangle{n}}\\
&=(\id\tens\underbrace{\Lambda\tens\dotsm\tens\Lambda}_n)\comp\alpha^{\vartriangle{n}}\\
&=(\id\tens\underbrace{\Lambda\tens\dotsm\tens\Lambda}_n)\comp(\id\tens\Delta^{(n-1)})\comp\alpha\\
&=\bigl(\id\tens\bigl[(\underbrace{\Lambda\tens\dotsm\tens\Lambda}_n)\comp\Delta^{(n-1)}\bigr]\bigr)\comp\alpha.
\end{split}
\]
\item If the \cst-algebra $\sB$ is commutative then $\KK_\beta$ is a classical group. This follows from Remark \ref{RemHI}\eqref{RemHI1}.
\item If $\HH$ is a compact quantum group and $\sB=\C(\HH)$ while $\beta$ is an action of $\HH$ then $\KK_\beta=\HH$ and $\Bar{\beta}=\beta$ (cf.~Remark \ref{RemHI}\eqref{RemHI2}).
\item Assume that $\beta$ has the property that the set
\[
\bigl\{(\omega\tens\id)\beta(m)\st{m}\in\sM,\;\omega\in\sM^*\bigr\}
\]
generates a strictly dense subalgebra of $\M(\sB)$. Then $\theta$ has strictly dense range, because
\[
(\omega\tens\id)\beta(m)=(\omega\tens\id)(\id\tens\theta)\Bar{\beta}(m)=\theta\bigl((\omega\tens\id)\Bar{\beta}(m)\bigr)
\]
for all $m\in\sM$ and $\omega\in\sM^*$. This means that if the family $\beta$ contains each of its members only once then $\beta$ is a ``subfamily'' of the family $\Bar{\beta}$ (cf.~\cite[Theorem 1.1(5)]{dkss}).
\item If the family $\beta$ is \emph{trivial}, i.e.
\[
\beta(m)=m\tens\I_\sB\in\M(\sM\tens\sB),\qquad{m}\in\sM,
\]
then it is easy to see that $\KK_\beta$ is the trivial group: $\C(\KK_\beta)=\CC$ and the map $\theta$ is of course, the opposite of surjective --- its image is $\CC\I\subset\M(\sB)$. Moreover, if $\pi$ is the canonical map $\C(\GG)\to\C(\KK_\beta)$ such that $\beta=(\id\tens\pi)\comp\alpha$, then for all $x\in\C(\GG)$ we have $\pi(x)=\epsilon(x)$, where $\epsilon$ is the counit of $\GG$.
\item The construction above yields another view-point on the notion of \emph{inner linearity} introduced in \cite{BanicaBichon}. Recall that a Hopf algebra $A$ is inner-linear if it admits an algebra morphism $\pi$ into a matrix algebra such that $A=A_{\pi}$. Thus if $\GG$ is a compact quantum group then the algebra $\Pol(\GG)$ is inner linear if $\GG$ is generated be a finite quantum family of invertible maps of some finite quantum space described by a pair $(\sM,\phi)$.

\end{enumerate}
\end{remark}


\section{Quantum increasing sequences of Curran and related problems}

In the last short section we explain the connection between quantum increasing sequences of Curran and problems studied in this article.

The notion of quantum increasing sequences was introduced by S.\,Curran in \cite{Curran}. They can be viewed as a quantum family of (invertible) maps from the set $\{1,\ldots,k\}$ to $\{1, \ldots, n\}$, where $k, n \in \NN$, $k \leq n$. We now recall the main definitions in \cite{Curran}, changing slightly the notation so that it is compatible with the rest of our article.

\begin{definition}[\cite{Curran}, Definition 2.1]
Let $k, n \in \NN$, $k \leq n$. The algebra of continuous functions on the space of quantum increasing sequences of length $k$ with values in $\{1, \ldots, n\}$, $\C (\II_{k,n})$, is the universal unital \cst-algebra generated by the elements $\{v_{ij}:i =1, \ldots,n, j=1,\ldots,k\}$ such that
\begin{rlist}
\item   $v_{ij}$ is an orthogonal projection for each $i\in\{1, \ldots,n\}, j\in \{1,\ldots,k\}$;
\item $\sum_{i=1}^n v_{ij} =1$ for each $ j\in \{1,\ldots,k\}$;
\item $v_{ij} v_{i'j'}=0$ if $i,i'\in\{1, \ldots,n\}, j,j'\in \{1,\ldots,k\}$, $j < j'$, $i \geq i'$.
\end{rlist}
\end{definition}

The algebra $\C(\II_{k,n})$  in \cite{Curran} is denoted $A_i(k,n)$. It is easy to check that if $(e_1, \ldots, e_n)$ and $(f_1, \ldots, f_k)$ denote respectively the canonical bases in $\CC^n$ and $\CC^k$ the formula
\[ \alpha(e_i) = \sum_{j=1}^k f_j \otimes v_{ij},\;\;\; i=1, \ldots, n,\]
defines a quantum family of invertible maps $\{1,\ldots,k\}$ to $\{1, \ldots, n\}$ indexed by $\II_{k,n}$. Similarly one can verify that the commutative version of the algebra $\C(\II_{k,n})$ is the algebra of functions on the space increasing sequences of length $k$ with values in $\{1, \ldots, n\}$ and that this identification is compatible with the map $\alpha$ defined above. For us it is going to be important that one can `complete' a quantum increasing sequence of  of length $k$ with values in $\{1, \ldots, n\}$ to a quantum permutation in $S_n^+$. By that statement we understand the fact that there exists an injective morphism between quantum spaces  $\II_{k,n}$ and $S_n^+$, which manifests itself as a surjective \cst-algebra morphism $\C(S_n^+) \to \C(\II_{k,n})$.

\begin{proposition}[\cite{Curran}, Proposition 2.5] \label{completion}
Let $k, n \in \NN$, $k \leq n$, and let $(p_{ij}:i =1, \ldots,n, j=1,\ldots,n)$, $(v_{ij}:i =1, \ldots,n, j=1,\ldots,k)$ denote respectively the canonical generators of $\C(S_n^+)$ and $\C(\II_{k,n})$. Put in addition $v_{00}=1$, $v_{i0}=v_{0i}=v_{i,k+1}=0$ for $i=1,\ldots,n$. Then the map
\[ p_{ij} \mapsto v_{ij}, \;\;\;\;i\in\{1, \ldots,n\}, j\in \{1,\ldots,k\},\]
\[ p_{i,k+m} \mapsto 0, \;\;\;\; m\in \{1,\ldots,n-k\},i\in\{1, \ldots,m-1\} \cup \{m+k+1, \ldots,n\}, \]
\[ p_{p+m,k+m} \mapsto \sum_{i=0}^{m+p-1} (v_{ip} - v_{i+1, p+1}), \;\;\;\;m\in\{1, \ldots,n-k\}, p\in \{0,\ldots,k\},\]
extends uniquely to a surjective unital $*$-homomorphism $\gamma: \C(S_n^+) \to \C(\II_{k,n})$.
\end{proposition}

Note here that the proof of the above uses also \cite[ Lemma 2.4 (ii)]{Curran}, where it is shown that $v_{ij}=0$ unless $j\leq i \leq n-k+j$.

The above proposition naturally leads to considering the quantum family of invertible maps on $\CC^n$ defined by the map $\kappa: \CC^n \to \CC^n \otimes \C(\II_{k,n})$, where
$\kappa = (\textup{id} \otimes \gamma) \circ \beta$ and $\beta$ is the canonical action of $S_n^+$ on $\CC^n$. Further Theorem \ref{genbeta} allows us to consider the quantum group $\KK_{k,n}$ (a quantum subgroup of $S_n^+$) generated by the family $\beta$. We will call it the quantum permutation group of $n$ elements generated by quantum increasing sequences of length $k$.

\begin{question} \label{questgen}
For what values of $k$ and $n$ quantum increasing sequences of length $k$ generate all quantum permutations in $S_n^+$? In other words, when $\KK_{k,n} = S_n^+$?
\end{question}

The answer to this question is sometimes negative for elementary reasons: for example when $k\in \{n-1,n\}$ then $\C(\II_{k,n})$ is in fact commutative and thus if $n>3$ we cannot have $\KK_{k,n} = S_n^+$. Further $\C(\II_{1,n})$ is also commutative for any $n \in \NN$, so in fact $\KK_{1,n} = S_n \subset S_n^+$. Indeed, by Remark \ref{Rem:genInvertible} $\KK_{1,n}$ is a classical subgroup of $S_n^+$, so also of $S_n$ -- and it is easy to see that the `completion' procedure formalised in Proposition \ref{completion} in this case realizes elements of $I_{1,n}$ as all cycles $(1, \ldots,l)$ with $l=1,\ldots,n$. The latter  generate $S_n$ as a group, which can be seen by simple induction.

On the other hand already when $k=2$, $n=4$ the question seems to be interesting and non-trivial. Note that the classical version of this question has a positive answer (in fact for an arbitrary $n\in \NN$, as long as $k=2$). By that we mean the fact the `completed' permutations arising from increasing sequences of length $2$ with values in $\{1, 2,3,4\}$ -- i.e.\ $\id$, the transposition $\tau_{2,3}$, the cycle $(2,4,3)$, the cycle $(1,2,3)$, the composition $\tau_{2,3} \tau_{1,3} \tau_{2,4}$ and the composition $\tau_{1,3} \tau_{2,4}$ -- generate $S_4$ as a group.

Consider then $k=2, n=4$.

\begin{proposition} \label{structure}
The \cst-algebra $\C(\II_{2,4})$ is isomorphic to $(\CC^2 \star \CC^2) \oplus \CC^2$.
\end{proposition}

\begin{proof}
The algebra $\C(\II_{2,4})$ is generated by two pairs of projections orthogonal to each other in each pair separately, say $p_1, p_2$ and $q_1, q_2$ such that we also have $p_1 q_1 =0$, $p_2 q_2 =0$ and $q_1 p_2=0$ (see the identification below); note that $q_1$ and $p_2$ commute with everything. Once $q_1$ and $p_2$ are set to $0$, the remaining two projections, $q_2$ and $p_1$ are free, and it is well-known that the algebra generated by two free projections is isomorphic to $\CC^2 \star \CC^2$ (or alternatively to $\cst(D_{\infty})$).
\end{proof}

Let us continue with more specific descriptions,  writing $v_{21}=p_1$, $ v_{31}=p_2$, $v_{22}=q_1$, $v_{32}=q_2$. Then the map $\gamma:\C(S_4^+) \to \C(\II_{2,4})$ from Proposition \ref{completion} is induced by the following magic unitary:

\[ \begin{bmatrix} 1 - p_1 - p_2 & 0 & p_1 +p_2 & 0 \\
p_1 & q_1 & 1 - p_1 - p_2 - q_1 & p_2 \\
p_2 & q_2 & q_1 & 1 - p_2  -q_1 - q_2  \\
0 & 1- q_1 - q_2 & 0 & q_1 +q_2
\end{bmatrix}
\]

This suggests that $\KK_{2,4}$ is a `minor augmentation of' $\widehat{D_{\infty}} \subset S_4^+$.

We believe that the most natural path to answer Question \ref{questgen} would lead via  \cite[Corollary 3.7]{Simeng}, quoted already at the end of Section 4, together with the Weingarten type formulas in \cite{Teosurvey} and in \cite[Proposition 4.4]{Curran}.

\section*{Appendix: Existence of universal quantum families}

\renewcommand{\thesection}{A}
\setcounter{proposition}{0}
\setcounter{equation}{0}

In this appendix we will give a proof of Theorem \ref{existenceQFM} which does not involve considering arbitrary families of relations. We begin with a version of Theorem \ref{existenceQFM} for $\C(\XX)=\cst(\FF_n)$ --- the full group \cst-algebra of the free group on $n$ generators. To lighten notation we will write $\sM$ for the finite dimensional \cst-algebra $\C(\YY)$.

\begin{proposition}\label{firststep}
Let $n \in \NN$ and let $\sM$ be a finite dimensional \cst-algebra. Then there exists a unital \cst-algebra $\sA_n$ and a quantum family of maps
\[
\Phi_n\colon\cst(\FF_n)\longrightarrow\sM\tens\sA_n
\]
such that for any \cst-algebra $\sD$ and any $\Psi\in\Mor\bigl(\cst(\FF_n),\sM\tens\sD\bigr)$ there exists a unique $\Lambda_n\in\Mor(\sA_n,\sD)$ such that
\[
\Psi=(\id\tens\Lambda_n)\comp\Phi_n.
\]
\end{proposition}

\begin{proof}
The finite dimensional \cst-algebra $\sM$ is of the form
\begin{equation}\label{decomp}
\sM\cong\bigoplus_{i=1}^NM_{m_i}(\CC).
\end{equation}
Let $\sA_n$ be the \cst-algebra generated by elements
\[
\bigl\{u^{i,p}_{k,l}\st{1\leq{i}\leq{N}},\;{1\leq{k,l}\leq{m_i},\;{1\leq{p}\leq{n}}}\bigr\}
\]
subject to the following relations: for each $p\in\{1,\dotsc,n\}$, $i\in\{1,\dotsc,N\}$ and any $k,l\in\{1,\dotsc,m_i\}$ we have
\[
\sum_{r=1}^{m_i}(u^{i,p}_{r,k})^*u^{i,p}_{r,l}=\delta_{k,l}\I=\sum_{r=1}^{m_i}u^{i,p}_{k,r}(u^{i,p}_{l,r})^*.
\]
In other words we are asking that for each $p$ and each $i$ the $m_i\times{m_i}$ matrix
\[
u^{i,p}=\begin{bmatrix}u^{i,p}_{1,1}&\dotsm&u^{i,p}_{1,m_i}\\
\vdots&\ddots&\vdots\\u^{i,p}_{m_i,1}&\dotsm&u^{i,p}_{m_i,m_i}\end{bmatrix}
\]
be unitary. Then we can define a unital $*$-homomorphism $\Phi_n\colon\cst(\FF_n)\to\sM\tens\sA_n$ by
\[
\Phi_n(v_p)=(u^{1,p},\dotsc,u^{N,p})\in\bigoplus_{i=1}^NM_{m_i}(\sA_n)\cong\sM\tens\sA_n
\]
where $v_1,\dotsc,v_n$ are the generators of $\FF_n$ ($\Phi_n$ exists by the universal property of $\cst(\FF_n)$).

Note now  that if $\widetilde{\sD}$ is a unital \cst-algebra containing elements
\[
\bigl\{w^{i,p}_{k,l}\st{1\leq{i}\leq{N}},\;{1\leq{k,l}\leq{m_i},\;{1\leq{p}\leq{n}}}\bigr\}
\]
such that the matrices
\[
w^{i,p}=\begin{bmatrix}w^{i,p}_{1,1}&\dotsm&w^{i,p}_{1,m_i}\\
\vdots&\ddots&\vdots\\w^{i,p}_{m_i,1}&\dotsm&w^{i,p}_{m_i,m_i}\end{bmatrix}
\]
are unitary for each $p$ and $i$, then there exists a unique unital $*$-homomorphism $\sA_n\to\widetilde{\sD}$ mapping $u^{i,p}_{k,l}$ to $w^{i,p}_{k,l}$ for all $i,p,k,l$.

Assume now that $\Psi\in\Mor\bigl(\cst(\FF_n),\sM\tens\sD\bigr)$ for some \cst-algebra $\sD$. Then the images of $v_1,\dotsc,v_n$ in $\M\bigl(\sM\tens\sD)\cong\sM\tens\M(\sD)$ under $\Psi$ are precisely collections $(w^{1,p},\dotsc,w^{N,p})$ of unitary $m_i\times{m_i}$ matrices with entries in $\M(\sD)$. It follows that there exists a unique unital $*$-homomorphism $\Lambda_n\colon\sA_n\to\M(\sD)$ with the property that $\Psi(x)=(\id\tens\Lambda_n)\Phi_n(x)$ for all $x\in\cst(\FF_n)$. Clearly $\Lambda_n$ is nondegenerate (it is, after all, a unital map).
\end{proof}

\begin{proof}[Proof of Theorem \ref{existenceQFM}]
Let us lighten notation by writing $\sM$ for the finite dimensional \cst-algebra $\C(\XX)$ and $\sC$ for the finitely generated unital \cst-algebra $\C(\YY)$. This means that we have a surjective $*$-homomorphims $\pi\colon\cst(\FF_n)\to\sC$. In what follows we will use the notation of Proposition \ref{firststep}. In particular the decomposition of $\sM$ into simple summands will be given by \eqref{decomp}.

Let $\sK$ denote the kernel of $\pi$ and define $\sJ$ as the ideal in $\sA_n$ generated by the set
\[
\bigl\{(\omega\tens\id)\Phi_n(x)\st{x}\in\sK,\;\omega\in\sM^*\bigr\}
\]
and let $\sA$ be the quotient \cst-algebra $\sA_n/\sJ$. Also let $\lambda$ be the quotient map $\sA_n\to\sA$. In particular we see that $\sA$ is a unital finitely generated \cst-algebra. In the notation from the statement of Theorem \ref{existenceQFM} $\sA$ will be $\C(\EE)$.

Now we shall show that there exists a unital $*$-homomorphism $\Phi\colon\sC\to\sM\tens\sA$ such that the diagram
\begin{equation}\label{bPhi}
\xymatrix{
{\cst}(\FF_n)\ar[rr]^-{\Phi_n}\ar[d]_{\pi}&&\sM\tens{\sA_n}\ar[d]^{\id\tens\lambda}\\
\sC\ar[rr]^-{\Phi}&&\sM\tens\sA}
\end{equation}
is commutative. To see this let us first note that if $x\in\cst(\FF_n)$ is such that $\pi(x)=0$, then writing $\Phi_n(x)$ as
\[
\Phi_n(x)=\sum_{j=1}^me_j\tens{x_j}
\]
with $(e_j)_{j=1,\dotsc,m}$ a basis of $\sM$ (so $m=\sum\limits_{i=1}^Nm_i^2$) we obtain $x_1,\dotsc,x_m\in\sJ$ because if $(\omega_j)_{j=1,\dotsc,m}$ is the dual basis to $(e_j)_{j=1,\dotsc,M}$ then $x_j=(\omega_j\tens\id)\Phi_n(x)$ and $x\in\ker\pi=\sK$. It follows that $\lambda(x_j)=0$ for all $j$ and consequently $(\id\tens\lambda)\Phi_n(x)=0$. This means that for any $c\in\sC$ the element $(\id\tens\lambda)\Phi_n(x)$, where $x$ is any lift of $c$ is independent of the choice of the lift. We denote this element of $\sM\tens\sA$ by $\Phi(c)$. It is very easy to check that so defined a map $\Phi\colon\sC\to\sM\tens\sA$ is a unital $*$-homomorphism making the diagram \eqref{bPhi} commutative.

We will now establish the universal property of $(\sA,\Phi)$, i.e.~we will show that for any \cst-algebra $\sD$ and any $\Psi\in\Mor(\sC,\sM\tens\sD)$ there exists a unique $\Lambda\in\Mor(\sA,\sD)$ such that $\Psi=(\id\tens\Lambda)\comp\Phi$.

Indeed,  let $\Psi\in\Mor(\sC,\sM\tens\sD)$ for some \cst-algebra $\sD$. By the universal property of $(\sA_n,\Phi_n)$ proved in Proposition \ref{firststep} there exists a unique $\Lambda_n\in\Mor(\sA_n,\sD)$ such that the diagram
\[
\xymatrix{
{\cst}(\FF_n)\ar[rr]^-{\Phi_n}\ar[d]_{\pi}&&
\sM\tens\sA\ar[d]^{\id\tens\Lambda_n}\\
\sC\ar[rr]^-{\Psi}&&\sM\tens\sD}
\]
commutes (note that this is obtained by applying the universal property of $(\sA_n,\Phi_n)$ to the map $\Psi\comp\pi\in\Mor\bigl(\cst(\FF_n),\sM\tens\sD)$).

Take $x\in\sK$. Since $(\id\tens\Lambda_n)\Phi_n(x)=\Psi\bigl(\pi(x)\bigr)=0$, we have
\[
\Lambda_n\bigl((\omega\tens\id)\Phi_n(x)\bigr)
=(\omega\tens\id)\bigl((\id\tens\Lambda_n)\Phi_n(x)\bigr)=0
\]
for any $\omega\in\sM^*$. It follows from the definition of $\sJ=\ker\lambda$, that $\Lambda_n$ vanishes on $\sJ=\ker\lambda$. Therefore there exists a unique unital $*$-homomorphism $\Lambda\colon\sA\to\M(\sD)$ such that
\begin{equation}\label{LaLa}
\Lambda_n=\Lambda\comp\lambda.
\end{equation}
Clearly the map $\Lambda$ is a morphism from $\sA$ to $\sD$.

From \eqref{LaLa} we obtain the commutative diagram
\[
\xymatrix
{
\sM\tens\sA_n\ar[rr]^-{\id\tens\Lambda_n}\ar[rd]_-{\id\tens\lambda}&&\sM\tens\sD\\
&\sM\tens\sA\ar[ur]_-{\id\tens\Lambda}
}
\]
Now combining this information with \eqref{bPhi} we see that the diagram
\begin{equation}\label{PP}
\xymatrix{{\cst(\FF_n)}\ar[rr]^-{\Phi_n}\ar[dd]_{\pi}&&\sM\tens\sA_n\ar[dl]_{\id\tens\lambda}
\ar[dd]^{\id\tens\Lambda_n}\\
&\sM\tens\sA\ar[rd]_-{\id\tens\Lambda}\\
\sC\ar[rr]_-{\Psi}\ar[ur]^-{\Phi}&&\sM\tens\sD}
\end{equation}
commutes. The diagram
\begin{equation}\label{commut}
\xymatrix{\sC\ar[rr]^-{\Phi}\ar@{=}[d]&&\sM\tens\sA\ar[d]^{\id\tens\Lambda}\\\sC\ar[rr]^-{\Psi}&&\sM\tens\sD}
\end{equation}
is simply a part of \eqref{PP}.

Note that the \cst-algebra $\sA$ is generated by the set
\[
\bigl\{(\omega\tens\id)\Phi(c)\st{c}\in\sC,\;\omega\in\sM^*\bigr\}.
\]
This implies that $\Lambda$ is the unique morphism from $\sA$ to $\sD$ making \eqref{commut} commutative.
\end{proof}

\end{document}